\documentclass[12pt,reqno]{article}

\usepackage[usenames]{color}
\usepackage[colorlinks=true,
linkcolor=webgreen, filecolor=webbrown,
citecolor=webgreen]{hyperref}

\definecolor{webgreen}{rgb}{0,.5,0}
\definecolor{webbrown}{rgb}{.6,0,0}

\usepackage{amssymb}
\usepackage{graphicx}
\usepackage{amscd}
\usepackage{lscape}

\usepackage{amsthm}
\newtheorem{theorem}{Theorem}

\newtheorem{proposition}[theorem]{Proposition}
\newtheorem{corollary}[theorem]{Corollary}

\theoremstyle{definition}

\newtheorem{example}[theorem]{Example}

\usepackage{float}

\usepackage{graphics,amsmath}
\usepackage{amsfonts}
\usepackage{latexsym}
\usepackage{epsf}

\newcommand{\seqnum}[1]{\href{http://oeis.org/#1}{\underline{#1}}}

\begin{document}

\begin{center}
\vskip 1cm{\LARGE\bf Riordan Pseudo-Involutions, Continued Fractions and Somos $4$ Sequences} \vskip 1cm \large
Paul Barry\\
School of Science\\
Waterford Institute of Technology\\
Ireland\\
\href{mailto:pbarry@wit.ie}{\tt pbarry@wit.ie}
\end{center}
\vskip .2 in

\begin{abstract} We define a three parameter family of Bell pseudo-involutions in the Riordan group. The defining sequences have generating functions that are expressible as continued fractions. We indicate that the Hankel transforms of the defining sequences, and of the $A$ sequences of the corresponding Riordan arrays, can be associated with Somos $4$ sequence. We give examples where these sequences can be associated with elliptic curves, and we exhibit instances where elliptic curves can give rise to associated Riordan pseudo-involutions. In the case of a particular one parameter family of elliptic curves, we show how we can associate to each such curve a unique Bell pseudo-involution.\end{abstract}

\section{Introduction}

The area of Riordan (pseudo) involutions has been the subject of much research in recent years \cite{PSI, B_Seq, Cons, Inv, Cohen, Luzon, Candice}. Recently, a sequence characterization of these involutions has emerged \cite{B_Seq}. This sequence is called the $\Delta$-sequence or the $B$-sequence. In this paper, we study a three parameter family of Riordan pseudo-involutions defined by a simply described $B$-sequence. We show that these pseudo-involutions are linked to Catalan defined generating functions, and are linked to Somos sequences and elliptic curves via the Hankel transforms of these generating functions. We show by example that it is possible to start with an appropriate elliptic curve and to derive from its equation an associated Riordan pseudo-involution.

The group of (ordinary) Riordan arrays \cite{Book, Survey, SGWW} is the set of lower triangular invertible matrices $(g(x), f(x))$ defined by two power series
$$g(x)=1+g_1 x+ g_2 x^2 + \cdots,$$  and
$$f(x)=f_1 x+ f_2 x^2 + \cdots,$$ such that the $(n,k$)-th element $t_{n,k}$ of the matrix is given by
$$t_{n,k}=[x^n] g(x)f(x)^k,$$ where $[x^n]$ is the functional that extracts the coefficient of $x^n$ from a power series.

The Fundamental Theorem of Riordan Arrays (FTRA) says that we have the law $$(g(x), f(x))\cdot h(x)=g(x)h(f(x)).$$

This is equivalent to the matrix represented by $(g(x), f(x))$ operating on the column vector whose elements are the expansion of the generating function $h(x)$. The resulting vector, regarded as a sequence, will have generating function $g(x)h(f(x))$.

The product of two Riordan arrays $(g(x),f(x))$ and $(u(x),v(x))$ is defined by
$$(g(x),f(x)) \cdot (u(x),v(x))=(g(x)u(f(x)), v(f(x))).$$
The inverse of the Riordan array $(g(x),f(x))$ is given by
$$(g(x), f(x))^{-1}=\left(\frac{1}{g(\bar{f}(x))}, \bar{f}(x)\right),$$ where
$\bar{f}(x)=\text{Rev}(f)(x)$ is the compositional inverse of $f(x)$. Thus $\bar{f}(x)$ is the solution $u(x)$ of the equation
$$f(u)=x$$ with $u(0)=0$.

The identity element is given by $(1,x)$ which as a matrix is the usual identity matrix.

With these operations the set of Riordan arrays form a group, called the Riordan group.

The Bell subgroup of Riordan arrays consists of those lower triangular invertible matrices defined by a power series
 $$g(x)=1+g_1 x+ g_2 x^2 + \cdots,$$ where the $(n,k$)-th element $t_{n,k}$ of the matrix is given by
 $$t_{n,k}=[x^n] g(x)(xg(x))^k = [x^{n-k}] g(x)^{k+1}.$$
A Bell pseudo-involution is a Bell array $(g(x), xg(x))$ such that the square of the Riordan array $(g(x),-xg(x))$ is the identity matrix. We shall call a generating function $g(x)$ for which this is true an \emph{involutory} generating function.

For a lower triangular invertible matrix $A$, the matrix $P_A=A^{-1}\bar{A}$ is called the production matrix of $A$, where $\bar{A}$ is the matrix $A$ with its first row removed. A matrix $A$ is a Riordan array if and only if $P_A$ takes the form
$$\left(
\begin{array}{cccccc}
 z_0 & a_0 & 0 & 0 & 0 & 0 \\
 z_1 & a_1 & a_0 & 0 & 0 & 0 \\
 z_2  & a_2  & a_1 & a_0 & 0 & 0 \\
 z_3  & a_3  & a_2 & a_1 & a_0 & 0 \\
 z_4  & a_4  & a_3  & a_2 & a_1 & a_0 \\
 z_5  & a_5  & a_4  & a_3 & a_2  &a_1 \\
\end{array}
\right).$$

The sequence that begins
$$z_0, z_1, z_2, z_3,\dots$$ is called the $Z$-sequence, while the sequence
$$a_0, a_1, a_2, a_3,\dots$$ is called the $A$-sequence. For a Riordan array $(g(x), f(x))$, we have
$$A(x)=\frac{x}{\bar{f}(x)} \quad \text{and}\quad Z(x)=\frac{1}{\bar{f}(x)}\left(1-\frac{1}{g(\bar{f}(x))}\right),$$
where $A(x)$ is the power series $a_0+a_1 x+ a_2 x^2 +\cdots$, and $Z(x)$ is the power series $z_0+z_1 x+ z_2 x+ \cdots$.

For a Riordan array $(g(x), f(x))$ to be an element of the Bell subgroup it is necessary and sufficient that
$$A(x)=1+xZ(x).$$
If $(g(x), x g(x))$ is a pseudo-involution, then we have that \cite{Cons}
$$A(x)=\frac{1}{g(-x)}.$$
We have the following result \cite{B_Seq}.
\begin{proposition} An element $(g(x), xg(x))$ of the Bell subgroup is a pseudo-involution if and only if there exists a sequence
$$b_0, b_1, b_2, \ldots$$ such that
$$t_{n+1,k}=t_{n,k-1}+\sum_{j \ge 0} b_j\cdot t_{n-j,k+j},$$ where
$d_{n,k-1}=0$ if $k=0$.
\end{proposition}
This sequence, when it exists, is called the $B$-sequence, or the $\Delta$-sequence, of the Riordan pseudo-involution.
The relationship between $A(x)$ and $B(x)$, when this latter exists, is given by \cite{Cons}
$$A(x)=1+xB\left(\frac{x^2}{A(x)}\right).$$

Sequences and triangles, where known, will be referenced by their $Annnnnn$ number in the On-Line Encyclopedia of Integer Sequences \cite{SL1, SL2}. All number triangles in this note are infinite in extent; where shown, a suitable truncation is used.
\section{A Bell pseudo-involution defined by continued fractions}
In this section, we consider the $B$-sequence with generating function given by
$$B(x)=\frac{a-cx}{1+bx}.$$
\begin{proposition} For the Bell pseudo-involution $(g(x), xg(x))$ with
$$B(x)=\frac{a-cx}{1+bx},$$ we have
$$A(x)=1+ax-\frac{x^3(ab+c)}{1+ax+bx^2} c\left(\frac{x^3(ab+c)}{(1+ax+bx^2)^2}\right),$$
and
$$g(x)=\frac{1}{1-ax-bx^2}c\left(\frac{-x^2(b+cx)}{(1-ax-bx^2)^2}\right),$$
where $$c(x)=\frac{1-\sqrt{1-4x}}{2x}$$ is the generating function of the Catalan numbers \seqnum{A000108}
$C_n=\frac{1}{n+1}\binom{2n}{n}$.
\end{proposition}
\begin{proof}
In order to solve for $A(x)$, we must solve the equation
$$u=1+x \frac{a-cx^2/u}{bx^2/u+1}$$ for $u(x)$.
We find that the appropriate branch is given by
$$u(x)=A(x)=\frac{1+ax-bx^2+\sqrt{1+2ax+(a^2+2b)x^2-2(ab+2c)x^3+bx^4}}{2}.$$
Using $c(x)=\frac{1-\sqrt{1-4x}}{2x}$ we can put this in the form
$$A(x)=1+ax-\frac{x^3(ab+c)}{1+ax+bx^2} c\left(\frac{x^3(ab+c)}{(1+ax+bx^2)^2}\right).$$
Now we have $$g(x)=\frac{1}{A(-x)}.$$ We find that
$$g(x)=\frac{2}{1-ax-bx^2+\sqrt{1-2ax+(a^2+2b)x^2+2(ab+2c)x^3+bx^4}},$$ or
$$g(x)=\frac{-1+ax+bx^2+\sqrt{1-2ax+(a^2+2b)x^2+2(ab+2c)x^3+b^2x^4}}{2x^2(cx+b)}.$$
This last expression can then be expressed as
$$g(x)=\frac{1}{1-ax-bx^2}c\left(\frac{-x^2(b+cx)}{(1-ax-bx^2)^2}\right).$$
\end{proof}

We now recall that $c(x)$ can be expressed as the continued fraction \cite{CFT, Wall}
$$c(x)=
\cfrac{1}{1-
\cfrac{x}{1-
\cfrac{x}{1-\cdots}}}.$$
Thus we have
\begin{corollary} For the pseudo-involution $(g(x), xg(x))$ with $B$-sequence given by
$$B(x)=\frac{a-cx}{1+bx},$$ we have that $g(x)$ can be expressed as the continued fraction
$$g(x)=
\cfrac{1}{1-ax-bx^2+
\cfrac{x^2(b+cx)}{1-ax-bx^2+
\cfrac{x^2(b+cx)}{1-ax-bx^2+\cdots}}}.$$
\end{corollary}

\begin{corollary} For the pseudo-involution $(g(x), xg(x))$ with $B$-sequence given by
$$B(x)=a+dx,$$ we have that $g(x)$ can be expressed as the continued fraction
$$g(x)=
\cfrac{1}{1-ax-
\cfrac{dx^3}{1-ax-
\cfrac{dx^3}{1-ax-\cdots}}}.$$
\end{corollary}

\begin{corollary} For the pseudo-involution $(g(x), xg(x))$ with $B$-sequence given by
$$B(x)=\frac{a}{1-bx},$$ we have that $g(x)$ can be expressed as the continued fraction
$$g(x)=
\cfrac{1}{1-ax+bx^2-
\cfrac{bx^2}{1-ax+bx^2-
\cfrac{bx^2}{1-ax+bx^2-\cdots}}}.$$
\end{corollary}

\section{Examples}
In this section, we examine some examples of the above sequences $g_n$, where $g(x)=\sum_{n=0}^{\infty}g_n x^n$ is such that $(g(x), xg(x))$ is a pseudo-involution. Thus we have $(g(x),-xg(x))^2=(1,x)$.
\begin{example}
We recall that for the pseudo-involution $(g(x), xg(x))$ with $B$-sequence given by
$$B(x)=a+dx,$$ we have that $g(x)$ can be expressed as the continued fraction
$$g(x)=
\cfrac{1}{1-ax-
\cfrac{dx^3}{1-ax-
\cfrac{dx^3}{1-ax-\cdots}}}.$$
Thus $$g(x)=\left(\frac{1}{1-ax},\frac{dx^3}{(1-ax)^2}\right)\cdot c(x)=\frac{1}{1-ax}c\left(\frac{dx^3}{(1-ax)^2}\right).$$
We then have
\begin{align*}
g_n&=\sum_{k=0}^{\lfloor \frac{n}{3} \rfloor} \binom{n-k}{n-3k}d^k a^{n-3k}C_k\\
&=\sum_{k=0}^n \frac{1}{n-k+3}\binom{\frac{2n+k}{3}}{\frac{n+2k}{3}}\binom{\frac{n+2k}{3}}{k}\left(2 \cos\left(\frac{2(n-k)\pi}{3}\right)+1\right)a^k d^{(n-k)/3}.\end{align*}
For instance, we have
$$\left(\begin{array}{c}g_0\\g_1\\g_2\\g_3\\g_4\\g_5\\g_6\\g_7\\\end{array}\right)=\left(
\begin{array}{cccccccc}
 1 & 0 & 0 & 0 & 0 & 0 & 0 & 0 \\
 0 & 1 & 0 & 0 & 0 & 0 & 0 & 0 \\
 0 & 0 & 1 & 0 & 0 & 0 & 0 & 0 \\
 d & 0 & 0 & 1 & 0 & 0 & 0 & 0 \\
 0 & 3 d & 0 & 0 & 1 & 0 & 0 & 0 \\
 0 & 0 & 6 d & 0 & 0 & 1 & 0 & 0 \\
 2 d^2 & 0 & 0 & 10 d & 0 & 0 & 1 & 0 \\
 0 & 10 d^2 & 0 & 0 & 15 d & 0 & 0 & 1 \\
\end{array}
\right)\left(\begin{array}{c}a_0\\a_1\\a_2\\a_3\\a_4\\a_5\\a_6\\a_7\\\end{array}\right).$$
The above matrix is an aeration of \seqnum{A060693}, which counts the number of Schr\"oder paths from $(0,0)$ to $(2n,0)$ having $k$ peaks.

We note that the Hankel transform $h_n=|g_{i+j}|_{0\le i,j \le n}$ begins
$$1, 0, - d^2, - d^4, 0, d^{10}, d^{14}, 0, - d^{24}, - d^{30}, 0,\ldots.$$
For $a=d=1$, we get the sequence that begins
$$ 1, 1, 1, 2, 4, 7, 13, 26, 52, 104, 212,\ldots$$
This is \seqnum{A023431}, which counts Motzkin paths of length $n$ with no $UD$'s and no $UU$'s.
For $a=1, d=2$ we get the sequence \seqnum{A091565}, that begins
$$1, 1, 1, 3, 7, 13, 29, 71, 163, 377, 913,\ldots.$$
For $a=2$, $d=1$, we get the sequence \seqnum{A091561} that begins
$$1, 2, 4, 9, 22, 56, 146, 388, 1048, 2869, 7942,\ldots.$$
The related sequence \seqnum{A152225} that begins
$$1,1, 2, 4, 9, 22, 56, 146, 388, 1048, 2869, 7942,\ldots$$ counts
the number of Dyck paths of semi-length $n$ with no peaks at height $0$ (mod $3$) and no valleys at height $2$ (mod $3$).
\end{example}
\begin{example}
When $c=0$ we obtain that for the pseudo-involution $(g(x), xg(x))$ with $B$-sequence given by
$$B(x)=\frac{a}{1+bx},$$ we have that $g(x)$ can be expressed as the continued fraction
$$g(x)=
\cfrac{1}{1-ax-bx^2+
\cfrac{bx^2}{1-ax-bx^2+
\cfrac{bx^2}{1-ax-bx^2+\cdots}}}.$$ In this case we have
$$g(x)=\left(\frac{1}{1-ax-bx^2}, \frac{-bx^2}{(1-ax-bx)^2}\right)\cdot c(x).$$
We find that
\begin{align*}g_n&=\sum_{k=0}^n (\sum_{j=0}^{n-2k}\binom{2k+j}{j}\binom{j}{n-2k-j}b^{n-2k-j}a^{2j+2k-n})(-b)^k C_k\\
&= \sum_{k=0}^n (\sum_{i=0}^{n-2k} \binom{n-i}{2k}\binom{n-2k-i}{i}b^i a^{n-2k-2i})(-b)^k C_k.\end{align*}

We have the following characterization of these sequences \cite{Gen, Conj}.
\begin{proposition} For the pseudo-involution $(g(x), xg(x))$ with $B$-sequence given by
$$B(x)=\frac{a}{1+bx},$$ we have that $g_n$ satisfies the recurrence
\begin{equation*}
g_n=
\begin{cases}
 1, & \text{if $n= 0$};\\
 a, & \text{if $n = 1$};\\
 a g_{n-1}+b g_{n-2}-b \sum_{k=0}^{n-2} g_k g_{n-2-k}, & \text{if $n>1$.}
 \end{cases}
  \end{equation*}
\end{proposition}
The Hankel transform of $g_n$ begins
$$1, 0, - a^2 b^2, - a^4 b^4, a^6 b^7, 0, - a^{12}b^{15}, a^{16}b^{20}, a^{20}b^{26}, 0, - a^{30}b^{40},\ldots.$$
For $a=1$, $b=-1$ we get the RNA sequence \seqnum{A004148} that begins
$$1, 1, 1, 2, 4, 8, 17, 37, 82,\ldots.$$
For $a=2$, $b=-1$ we get the sequence \seqnum{A187256}, which begins
$$1, 2, 4, 10, 28, 82, 248, 770, 2440,\ldots.$$  This counts the number of peakless Motzkin paths of length $n$, assuming that the $(1,0)$ steps come in $2$ colors (Emeric Deutsch).
\end{example}
\begin{example} In the general case, we have that
$$g(x)=\frac{1}{1-ax-bx^2}c\left(\frac{-x^2(b+cx)}{(1-ax-bx^2)^2}\right).$$
One expansion of this gives us
$$g_n=\sum_{k=0}^n (\sum_{i=0}^n \binom{k}{i}c^i b^{k-i} \sum_{m=0}^{n-2k-i} \binom{n-i-m}{n-2k-i-m}\binom{n-2k-i-m}{m}b^ma^{n-2k-2m-i})(-1)^k C_k.$$
For $a=2$, $b=-1$ and $c=1$, we obtain the sequence $g_n(2,-1,1)$ \seqnum{A105633} that begins
$$	1, 2, 4, 9, 22, 57, 154, 429, 1223, 3550, 10455,\ldots.$$ This sequence counts the number of Dyck paths of semi-length $n+1$ avoiding $UUDU$. We note that the related sequence that begins
$$	1,1, 2, 4, 9, 22, 57, 154, 429, 1223, 3550, 10455,\ldots $$ has generating function given by
$$
\cfrac{1}{1-x-
\cfrac{x^2}{1-
\cfrac{x}{1-x-
\cfrac{x^2}{1-
\cfrac{x}{1-x-\cdots}}}}}.
$$
The general term of this sequence is given by
$$\sum_{k=0}^{\lfloor \frac{n}{2} \rfloor} \sum_{j=0}^{n-k} \binom{n-k}{j}N_{j,k}$$
where $(N_{n,k})$ is the Narayana triangle \seqnum{A090181} with
$$N_{n,k}=\frac{1}{n-k+1} \binom{n}{k}\binom{n-1}{n-k},$$ which begins
$$\left(
\begin{array}{ccccccc}
 1 & 0 & 0 & 0 & 0 & 0 & 0 \\
 0 & 1 & 0 & 0 & 0 & 0 & 0 \\
 0 & 1 & 1 & 0 & 0 & 0 & 0 \\
 0 & 1 & 3 & 1 & 0 & 0 & 0 \\
 0 & 1 & 6 & 6 & 1 & 0 & 0 \\
 0 & 1 & 10 & 20 & 10 & 1 & 0 \\
 0 & 1 & 15 & 50 & 50 & 15 & 1 \\
\end{array}
\right).$$
Equivalently, this sequence is equal to the diagonal sums of the matrix product
$$\left(\binom{n}{k}\right)\cdot (N_{n,k}),$$ where this product matrix \seqnum{A130749} begins
$$\left(
\begin{array}{ccccccc}
 1 & 0 & 0 & 0 & 0 & 0 & 0 \\
 1 & 1 & 0 & 0 & 0 & 0 & 0 \\
 1 & 3 & 1 & 0 & 0 & 0 & 0 \\
 1 & 7 & 6 & 1 & 0 & 0 & 0 \\
 1 & 15 & 24 & 10 & 1 & 0 & 0 \\
 1 & 31 & 80 & 60 & 15 & 1 & 0 \\
 1 & 63 & 240 & 280 & 125 & 21 & 1 \\
\end{array}
\right).$$
The inverse binomial transform of $g_n(2,-1,1)$ is the sequence \seqnum{A007477} that begins
$$1, 1, 1, 2, 3, 6, 11, 22, 44, 90, 187, 392,\ldots.$$ This sequence counts the number of Dyck $(n+1)$-paths containing no $UDU$s and no sub-paths of the form $UUPDD$ where $P$ is a nonempty Dyck path (David Callan).
\end{example}
\section{Hankel transforms and Somos $4$ sequences}

The sequence $g_n(a,b,c)$ begins
$$1, a, a^2, a^3 - ab - c, a^4 - 3a^2b - 3ac, a^5 - 6a^3b - 6a^2c + ab^2 + bc,\ldots.$$
The Hankel transform of $g_n(a,b,c)$ begins
$$1, 0, - a^2b^2 - 2abc - c^2, - a^4b^4 - 4a^3b^3c - 6a^2b^2c^2 - 4abc^3 - c^4, \ldots.$$
Proceeding numerically, we can conjecture that the sequence $t_n$ that begins
$$- a^2b^2 - 2abc - c^2, - a^4b^4 - 4a^3b^3c - 6a^2b^2c^2 - 4abc^3 - c^4, \ldots$$ is a $((ab+c)^2, b(ab+c)^2)$ Somos $4$ sequence. By this we mean that
$$t_n=\frac{(ab+c)^2 t_{n-1} t_{n-3}+b(ab+c)^2 t_{n-2}^2}{t_{n-4}}.$$
\begin{example} We take $a=2$, $b=-2$ and $c=3$. We have
$$g(x)=\frac{1}{1-2x+2x^2}c\left(\frac{x^2(2-3x)}{(1-2x+2x^2)^2}\right)=\frac{1-2x+2x^2-\sqrt{1-4x+4x^3+4x^4}}{2x^2(2-3x)}.$$ This expands to give the sequence $g_n(2,-2,3)$ that begins
$$1, 2, 4, 9, 22, 58, 162, 472, 1418, 4357, 13618,\ldots.$$
The Hankel transform of $g_n(3,-2,3)$ begins
$$1, 0, -1, -1, -2, -3, 5, 28, 67, 411, 506,\ldots.$$
Now the sequence $$ -1, -1, -2, -3, 5, 28, 67, 411, 506,\ldots$$
is a $(1,-2)$ Somos $4$ sequence, associated with the elliptic curve $y^2 + y = x^3 + 3·x^2 + x$.
\end{example}
\begin{example} The sequence $g_n(-1,2,1)$ has generating function
$$g(x)=\frac{1}{1+x-2x^2}c\left(\frac{-x^2(2+x)}{(1+x-2x^2)^2}\right)=\frac{\sqrt{1+2x+5x^2+4x^4}-2x^2+1}{2x^2(x+2)}.$$ The sequence $g_n(-1,2,1)$ begins
$$1, -1, 1, 0, -2, 3, 1, -12, 20, 4, -84,\ldots.$$
It has a Hankel transform that begins
$$1, 0, -1, -1, 2, -1, -9, 16, 73, 145, -1442,\ldots.$$
Here, the sequence $-1, -1, 2, -1, -9, 16, 73, 145, -1442,\ldots$ is a
$(1,2)$ Somos $4$ sequence. It is related to \seqnum{A178075}, which is the $(1,2)$ Somos $4$ sequence that begins
$$1, 1, -2, 1, 9, -16, -73, -145, 1442,\ldots.$$
\end{example}
\begin{example} The sequence $g_n(-1,-2,-1)$ has generating function
$$g(x)=\frac{1}{1+x+2x^2}c\left(\frac{x^2(2+x)}{(1+x+2x^2)^2}\right)=\frac{1+x+2x^2-\sqrt{1+2x-3x^2+4x^4}}{2x^2(x+2)}.$$ The sequence begins
$$1, -1, 1, -2, 4, -9, 21, -50, 122, -302, 758,\ldots$$ and its Hankel transform begins
$$1, 0, -1, -1, -2, -1, 7, 16, 57, 113, -670,\ldots.$$
The sequence $$ -1, -1, -2, -1, 7, 16, 57, 113, -670,\ldots$$ is a $(1,-2)$ Somos $4$ sequence. Apart from signs, this is \seqnum{A178622}, which is associated with the elliptic curve $y^2-3xy-y=x^3-x$.
In fact, we can show that the $(1,-2)$ Somos $4$ sequence $1, 1, 2, 1, -7, -16, -57,\ldots$ can be described as the Hankel transform of the sequence that begins
$$1, 0, 1, -1, 4, -10, 30, -84, 237, -653, 1771, -4699, 12173, -30625,\ldots$$ with generating function
$$f(x)=\frac{2x}{\sqrt{1+6x+9x^2-4x^3-8x^4}-x-1}.$$
Note that many other sequences can have the same Hankel transform.

The relationship between $g(x)$ and $f(x)$ is given by
$$g(x)=\left(\frac{1}{1+2x}, \frac{-x}{1+2x}\right)\cdot \frac{f(x)(1+2x)-1}{f(x)x(3x+2)}.$$
\end{example}

\section{From elliptic curve to Riordan pseudo-involution}
In this section, we reprise the last example to make explicit the steps that lead from the elliptic curve given  by $$y^2-3xy-y=x^3-x$$ to the Riordan pseudo-involution defined by $g_n(-1,-2,-1)$.
The first step is to solve the quadratic equation
$$y^2-3xy-y=x^3-x$$
for $y$. The branch that we require is given by
$$\frac{1+3x-\sqrt{1+2x+9x^2+4x^3}}{2}.$$
This expands to give a sequence that begins
$$0,1,-2,1,3,-7,\ldots.$$
We must discard the first two terms, to get the generating function
$$\left(\frac{1+3x-\sqrt{1+2x+9x^2+4x^3}}{2}-x\right)/x^2=\frac{1+x-\sqrt{1+2x+9x^2+4x^3}}{2x^2}.$$
We now form the fraction
$$\frac{1}{1-x-x^2\left(\frac{1+x-\sqrt{1+2x+9x^2+4x^3}}{2x^2}\right)},$$ to get
the generating function
$$\frac{2}{1-3x+\sqrt{1+2x+9x^2+4x^3}}.$$
We now revert the generating function $\frac{2x}{1-3x+\sqrt{1+2x+9x^2+4x^3}}$ to obtain the generating function
$$\frac{1+3x-\sqrt{1+6x+9x^2-4x^3-8x^4}}{2x^2}.$$
Finally, we form the generating function
$$\frac{1}{1-x^2\left(\frac{1+3x-\sqrt{1+6x+9x^2-4x^3-8x^4}}{2x^3}\right)}$$ to arrive at
$$f(x)=\frac{2x}{\sqrt{1+6x+9x^2-4x^3-8x^4}-x-1}.$$
The sought after involutory generating function is now obtained by
$$g(x)=\left(\frac{1}{1+2x}, \frac{-x}{1+2x}\right)\cdot \frac{f(x)(1+2x)-1}{f(x)x(3x+2)}.$$
\begin{figure}
\begin{center}
\includegraphics[height=80mm,width=0.7\textwidth]{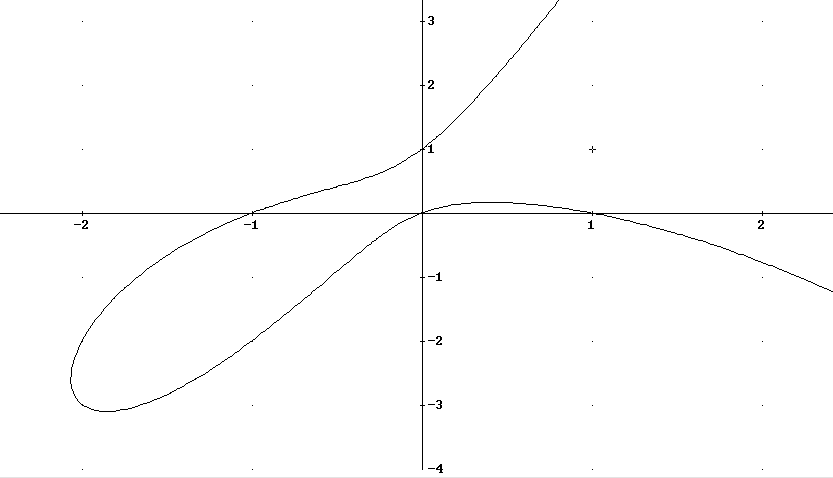}
\end{center}
\caption{The elliptic curve $y^2-3xy-y=x^3-x$}
\end{figure}

\begin{example} Inspired by the last section, we now start with the elliptic curve
$$y^2 - 2xy - y = x^3 - x$$ and seek to produce an involutary power series $g(x)$.
We begin as before by solving the quadratic in $y$ to get
$$\frac{1+2x-\sqrt{1+4x^2+4x^3}}{2}$$ which expands to give a sequence that begins $0,1,-1,-1,1,\ldots$.
We now form
$$\left(\frac{1+2x-\sqrt{1+4x^2+4x^3}}{2}-x\right)/x^2=\frac{1-\sqrt{1+4x^2+4x^3}}{2x^2}.$$
We now form the fraction
$$\frac{1}{1-x-x^2\left(\frac{1-\sqrt{1+4x^2+4x^3}}{2x^2}\right)}=\frac{2}{1-2x+\sqrt{1+4x^2+4x^3}}.$$
We revert the generating function $\frac{2x}{1-2x+\sqrt{1+4x^2+4x^3}}$ to obtain the generating function
$$\frac{1+2x-\sqrt{1+4x+4x^2-4x^3-4x^4}}{2x^2}.$$
We next from the generating function
$$\frac{1}{1-x^2\left(\frac{1+2x-\sqrt{1+4x+4x^2-4x^3-4x^4}}{2x^3}\right)}$$ to get
$$f(x)=\frac{1+\sqrt{1+4x+4x^2-4x^3-4x^4}}{1+x-x^2-x^3}.$$
We now let
$$g(x)=\left(\frac{1}{1+x}, \frac{-x}{1+x}\right)\cdot \frac{f(x)(1+x)-1}{f(x)x(2x+1)}.$$
Thus we arrive at
$$g(x)=\frac{1+x^2-\sqrt{1-2x^2+4x^3+x^4}}{2x^2(1-x)}=\frac{1}{1+x^2}c\left(\frac{x^2(1-x)}{(1+x^2)^2}\right).$$
This means that the sequence found is the involutory sequence $g_n(0,-1,1)$.

This sequence begins
$$1, 0, 0, -1, 0, -1, 2, -1, 5, -6, 9, -22, 28, -57, 104, -163, \ldots$$ and its Hankel transform begins
$$1, 0, -1, -1, -1, 1, 2, 3, 1, -7, -11,\ldots.$$
The sequence $$1, 1, 1, -1, -2, -3, -1, 7, 11,\ldots$$ is the $(1,-1)$ Somos $4$ sequence \seqnum{A050512} which is associated with the curve $$y^2-2xy-y=x^3-x$$ in the following way. This Somos sequence is $(-1)^{\binom{n}{2}}$ times the Hankel transform of the sequence with g.f.
$$
\cfrac{1}{1-
\cfrac{x^2}{1+
\cfrac{x^2}{1-
\cfrac{x^2}{1-
\cfrac{2x^2}{1+
\cfrac{(3/4)x^2}{1+
\cfrac{(2/9)x^2}{1+\cdots}}}}}}},$$
 where
$$-1,1,2,-3/4,-2/9,21,\ldots$$
 are the $x$-coordinates of the multiples of $z=(0,0)$ on the elliptic curve E:$y^2-2xy-y=x^3-x$.

The inverse binomial transform of $g_n$ begins
$$ 1, -1, 1, -2, 5, -12, 29, -72, 182, -466, 1207\ldots.$$ This is essentially \seqnum{A025273} or \seqnum{A217333}. The sequence $g_n$ is the partial sum sequence of the sequence that begins
$$1, -1, 0, -1, 1, -1, 3, -3, 6, -11, 15, -31, \ldots.$$
This is the alternating sign version of \seqnum{A025250}, whose binomial transform is essentially \seqnum{A025273}.
\begin{figure}
\begin{center}
\includegraphics[height=90mm,width=0.8\textwidth]{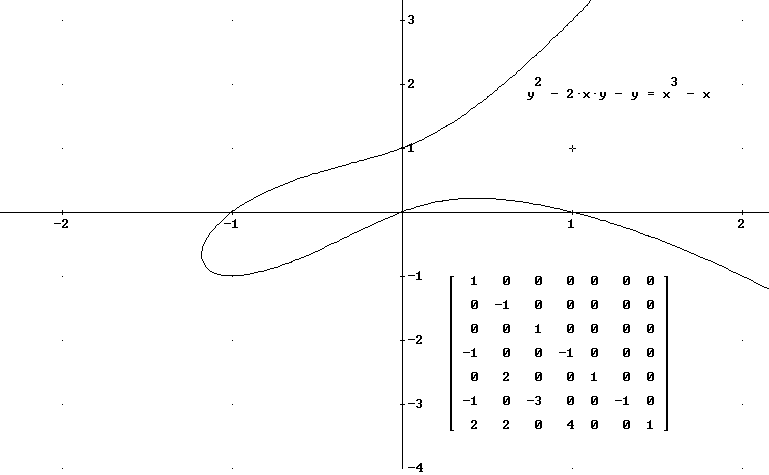}
\end{center}
\caption{The elliptic curve $y^2-2xy-y=x^3-x$ and its associated Riordan pseudo-involution}
\end{figure}
\end{example}
\section{The $A$-sequence and Somos $4$ sequences}
We recall that for a Bell pseudo-involution $(g(x), xg(x))$ for which
$$B(x)=\frac{a-cx}{1+bx},$$ we have
$$A(x)=1+ax-\frac{x^3(ab+c)}{1+ax+bx^2} c\left(\frac{x^3(ab+c)}{(1+ax+bx^2)^2}\right).$$
The Hankel transform of the expansion of $A(x)$ is not a Somos sequence, so we look at the element given by
$$\frac{1}{1+ax+bx^2}  c\left(\frac{x^3(ab+c)}{(1+ax+bx^2)^2}\right).$$ We can express this generating function as the continued fraction
$$
\cfrac{1}{1+ax+bx^2-
\cfrac{(ab+c)x^3}{1+ax+bx^2-
\cfrac{(ab+c)x^3}{1+ax+bx^2-\cdots}}}.$$
This expands to give a sequence that begins
$$1, -a, a^2 - b, - a^3 + 3ab + c, a^4 - 6a^2b - 3ac + b^2,\ldots$$ with a Hankel transform that begins
$$1, -b, - abc - c^2, - a^3b^3c + a^2b^2(b^3 - 3c^2) + abc(2b^3 - 3c^2) + c^2(b^3 - c^2),\ldots.$$
Once again, we can conjecture that this Hankel transform is a $((ab+c)^2, b(ab+c)^2)$ Somos $4$ sequence.

\begin{example}
For $(a,b,c)=(-1,1,2)$, we obtain the sequence (essentiallay \seqnum{A025258} that begins
$$1, 1, 0, 0, 2, 3, 1, 2, 11, 17, 12,\ldots$$ with generating function
$$\frac{1}{1-x+x^2}c\left(\frac{x^3}{(1-x+x^2)^2}\right)=\frac{1-x+x^2-\sqrt{1-2x+3x^2-6x^3+x^4}}{2x^3}.$$
This sequence has a Hankel transform that begins
$$1, -1, -2, -1, 5, 9, -8, -41, -61, 241,\ldots.$$
This is a $(1,1)$ Somos $4$ sequence (essentially \seqnum{A178627}) associated with the elliptic curve $$y^2 + xy - y = x^3 - x^2 + x.$$
\end{example}
\begin{example}
For $(a,b,c)=(-1,-1,-2)$ we obtain the sequence that begins
$$1, 1, 2, 2, 2, -1, -7, -20, -37, -53, -40, 49, 301, \ldots$$ with generating function
$$\frac{1}{1-x-x^2}c\left(\frac{-x^3}{(1-x-x^2)^2}\right)=\frac{-1+x+x^2+\sqrt{1-2x-x^2+6x^3+x^4}}{2x^3}.$$
This sequence has a Hankel transform that begins
$$1, 1, -2, -3, -7, 5, 32, 83, 87, -821,\ldots,$$, which is a $(1,-1)$ Somos $4$ sequence.
\end{example}
\begin{example} We take $(a,b,c)=(1,2,-1)$. Thus we obtain the sequence that begins
$$1, -1, -1, 4, -4, -5, 23, -28, -28, 164, -232, -166, \ldots$$ with generating function
$$\frac{1}{1+x+2x^2}c\left(\frac{x^3}{(1+x+2x^2)^2}\right)=\frac{1+x+2x^2-\sqrt{1+2x+5x^2+4x^4}}{2x^3}.$$
This sequence has a Hankel transform that begins
$$1, -2, 1, 9, -16, -73, -145, 1442, 3951, -49121,\ldots.$$
This is a $(1,2)$ Somos $4$ sequence, essentially \seqnum{A178075}.
\end{example}
\begin{example} We let $(a,b,c)=(-1,-2,-1)$. We obtain the sequence that begins
$$1, 1, 3, 6, 14, 33, 79, 194, 482, 1214, 3090, 7936, 20544, \ldots$$ with generating function
$$\frac{1}{1-x-2x^2}c\left(\frac{x^3}{(1-x-2x^2)^2}\right)=\frac{1-x-2x^2-\sqrt{1-2x-3x^2+4x^4}}{2x^3}.$$
This sequence has a Hankel transform that begins
$$1, 2, 1, -7, -16, -57, -113, 670, 3983, 23647,\ldots.$$
This is a $(1,-2)$ Somos $4$ sequence, essentially \seqnum{A178622}, which is associated with the elliptic curve
$y^2-3xy-y=x^3-x$.
We note further that the sequence that begins
$$1, 2, 1, 1, 3, 6, 14, 33, 79, 194, 482, 1214, 3090, 7936, 20544, \ldots$$ or \seqnum{A025243} counts the number of Dyck $(n-1)$-paths that contain no $DUDU$'s and no $UUDD$'s for $n \ge 3$.
\end{example}
\section{Elliptic pseudo-involutions}
In this section, we consider the methods outlined above, as applied to a particular one parameter family of elliptic curves. We obtain a result concerning what may be called ``elliptic'' pseudo-involutions in the Riordan group, as each such pseudo-involution is associated in a unique way with an elliptic curve of the type discussed below.
Prior to this, we need to establish the following result.
\begin{proposition} The generating function
$$g(x)=\frac{1}{1-ax-bx^2}c\left(\frac{-x^2(b+cx)}{(1-ax-bx^2)^2}\right)$$ is involutory.
\end{proposition}
\begin{proof} We must establish that
$$(g(x), -xg(x))^{-1}=(g(x), - xg(x)).$$
Now $$(g(x), -xg(x))^{-1} = \left(\frac{1}{g\left(\text{Rev}(-xg(x))\right)}, \text{Rev}(-xg(x))\right).$$
Thus a first requirement is to show that
$$\text{Rev}(-xg(x))=-xg(x).$$
This follows from solving the equation
$$\frac{-u}{1-au-bu^2}c\left(\frac{-u^2(b+cu)}{(1-au-bxu^2)^2}\right)=x,$$ where we take the solution that satisfies $u(0)=0$.
We next require that
$$g(x) = \frac{1}{g\left(\text{Rev}(-xg(x))\right)},$$ or that
$$g(x) g(-xg(x))=1.$$
Equivalently, we must show that
$$xg(x) g(-xg(x)) =x.$$
Now
\begin{align*}
x &=(\text{Rev}(-xg(x)))(-xg(x))\\
&=(-xg(x))(-xg(x))\\
&= xg(x)g(-xg(x)).\end{align*}
\end{proof}

\begin{proposition}
The elliptic curve
$$E: y^2-a xy-y = x^3-x$$ defines a pseudo-involution $(g(x), xg(x))$ in the Riordan group whose $B$-sequence is given by $$B(x)=\frac{2-a+(1-3a+a^2)x}{1+(1-a)x}.$$
\end{proposition}
\begin{proof}
We solve the quadratic (in $y$) given by
$$ y^2-a xy-y = x^3-x$$ to obtain
$$y=\frac{1+ax-\sqrt{1+2(a-2)x+a^2x^2+4x^3}}{2}.$$
This expands to a sequence that begins
$$0,1,1-a,1-3a+a^2,\ldots.$$
We strip out the first two terms, giving the generating function
$$\left(\frac{1+ax-\sqrt{1+2(a-2)x+a^2x^2+4x^3}}{2}-x\right)/x^2=\frac{1-(2-a)x-\sqrt{1+2(a-2)x+a^2x^2+4x^3}}{2x^2}.$$
We now form the fraction
$$\frac{1}{1-x-x^2\left(\frac{1-(2-a)x-\sqrt{1+2(a-2)x+a^2x^2+4x^3}}{2x^2}\right)}, $$ which simplifies to give
$$\frac{2}{\sqrt{1+2(a-2)x+a^2x^2+4x^3}}.$$
We now revert the generating function $\frac{2x}{\sqrt{1+2(a-2)x+a^2x^2+4x^3}}$ and divide the result by $x$ to get
$$\frac{1+ax-\sqrt{1+2ax+a^2x^2-4x^3+4(1-a)x^4}}{2x^3}.$$
We let
$$f(x)=\frac{1}{1-x^2\left(\frac{1+ax-\sqrt{1+2ax+a^2x^2-4x^3+4(1-a)x^4}}{2x^3}\right)},$$ or
$$f(x)=\frac{2x}{\sqrt{1+2ax+a^2x^2-4x^3+4(1-a)x^4}+(2-a)x-1}.$$
Finally, we form
$$g(x)=\left(\frac{1}{1+(a-1)x}, \frac{-x}{1+(a-1)x}\right)\cdot \frac{f(x)(1+(a-1)x)-1}{xf(x)(ax+a-1)}.$$
This gives us
$$g(x)=\frac{\sqrt{1+2(a-2)x+(a^2-6a+6)x^2+2a(3-a)x^3+(a-1)^2x^4}+(1-a)x^2+(2-a)x-1}{2x^2((a^2-3a+1)x+a-1)}.$$
This can now be put in the form
$$g(x)=\frac{1}{1-(2-a)x-(1-a)x^2}c\left(\frac{-x^2((1-a)-(1-3a+a^2)x)}{(1-(2-a)x-(1-a)x^2)^2}\right).$$
Comparing this with $$\frac{1}{1-\alpha x - \beta x^2}c\left(\frac{-x^2(\beta+\gamma x)}{(1-\alpha x - \beta x^2)^2}\right)$$ we see that this shows that $g(x)$ is an involutory generating function associated with the $B$-sequence given by
$$B(x)=\frac{2-a+(1-3a+a^2)x}{1+(1-a)x}.$$
\end{proof}
The $B$ sequence with generating function $B(x)=\frac{2-a+(1-3a+a^2)x}{1+(1-a)x}$ begins
$$2 - a, -1, -(a-1), - (a - 1)^2, -(a - 1)^3, - (a - 1)^4, -(a - 1)^5, - (a - 1)^6, -(a - 1)^7, - (a - 1)^8, \ldots.$$
The sequence $g_n$ begins
$$1, 2 - a, a^2 - 4a + 4, - a^3 + 6a^2 - 12a + 7, a^4 - 8a^3 + 24a^2 - 29a + 10,\ldots,$$ and it has a Hankel transform $|g_{i+j}|_{0 \le i,j \le n}$ which begins
$$1, 0, -1, -1, 1 - a, - a^2 + 3a - 1,\ldots.$$
The sequence $$1,1,a-1,a^2-3a+1,-a^3+4a^2-6a+2,\ldots$$ is in fact the Hankel transform of the sequence whose generating function is $f(x)$. This Hankel transform is a $(1,1-a)$ Somos $4$ sequence \cite{Gen, Conj}.

\begin{center}
\begin{tabular} {|c|c|c|}
\hline $a$ & $b_n$ & $g(x)$ \\\hline
 $0$ & $2,-1,1,-1,1,-1,\ldots$ & $\frac{\sqrt{1-4x+6x^2+x^4}+x^2+2x-1}{2x^2(1-x)}$ \\\hline
 $1$ & $1,-1,0,0,0,\ldots$ & $\frac{\sqrt{1-2x+x^2+4x^3}+x-1}{2x^3}$  \\\hline
 $2$ & $0,-1,-1,-1,\ldots$ & $\frac{\sqrt{1-2x^2+4x^3+x^4}-x^2-1}{2x^2(x-1)}$  \\\hline
 $3$ & $-1,-1,-2,-4,-8,\ldots$ & $\frac{1+x+2x^2-\sqrt{1+2x-3x^2+4x^4}}{2x^2(x+2)}$ \\\hline
 $4$ & $-2,-1,-3,-9,-27,\ldots$ & $\frac{1+2x+3x^2-\sqrt{1+4x-2x^2-8x^3+9x^4}}{2x^2(5x+3)}$  \\\hline
 $5$ & $-3,-1,-4,-16,\ldots$ & $\frac{1+3x+4x^2-\sqrt{1+6x+x^2-20x^3+16x^4}}{2x^2(11x+4)}$ \\\hline
\end{tabular}
\end{center}

\begin{example} We take the case of $a=-3$. Thus we start with the elliptic curve
$$E: y^2+3xy-y=x^3-x.$$
We find that
$$g(x)=\frac{\sqrt{1-10x+33x^2-36x^3+16x^4}+4x^2+5x-1}{2x^2(4-19x)},$$ which expands to give the sequence $g_n$ that begins
$$1, 5, 25, 124, 610, 2979, 14457, 69784, 335330, 1605334, 7662014,\ldots.$$
The corresponding Riordan pseudo-involution then begins
$$\left(
\begin{array}{cccccccc}
 1 & 0 & 0 & 0 & 0 & 0 & 0 & 0 \\
 5 & 1 & 0 & 0 & 0 & 0 & 0 & 0 \\
 25 & 10 & 1 & 0 & 0 & 0 & 0 & 0 \\
 124 & 75 & 15 & 1 & 0 & 0 & 0 & 0 \\
 610 & 498 & 150 & 20 & 1 & 0 & 0 & 0 \\
 2979 & 3085 & 1247 & 250 & 25 & 1 & 0 & 0 \\
 14457 & 18258 & 9300 & 2496 & 375 & 30 & 1 & 0 \\
 69784 & 104580 & 64512 & 21755 & 4370 & 525 & 35 & 1 \\
\end{array}
\right),$$ which has a production matrix that begins
$$\left(
\begin{array}{ccccccc}
 5 & 1 & 0 & 0 & 0 & 0 & 0 \\
 0 & 5 & 1 & 0 & 0 & 0 & 0 \\
 -1 & 0 & 5 & 1 & 0 & 0 & 0 \\
 5 & -1 & 0 & 5 & 1 & 0 & 0 \\
 -21 & 5 & -1 & 0 & 5 & 1 & 0 \\
 84 & -21 & 5 & -1 & 0 & 5 & 1 \\
 -326 & 84 & -21 & 5 & -1 & 0 & 5 \\
\end{array}
\right).$$ For this case, we have
$$B(x)=\frac{5+19x}{1+4x}.$$
The Hankel transform of $g_n$ begins
$$1, 0, -1, -1, 4, -19, -83, -1112, 12171,\ldots$$ corresponding to the $(1,4)$ Somos $4$ sequence that begins
$$1, 1, -4, 19, 83, 1112, -12171,\ldots.$$
\begin{figure}
\begin{center}
\includegraphics[height=80mm,width=0.7\textwidth]{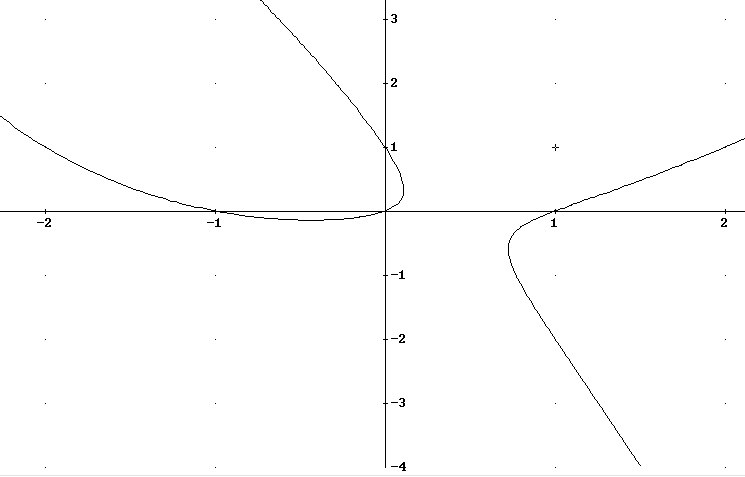}
\end{center}
\caption{The elliptic curve $y^2+3xy-y=x^3-x$}
\end{figure}
\end{example}
\begin{example} When $a=0$, we find that
$$g(x)=\left(\frac{1}{1-2x-x^2}, \frac{x^2(x-1)}{(1-2x-x^2)^2}\right)\cdot c(x).$$
This expands to give the sequence $g_n$ that begins
$$1, 2, 4, 7, 10, 9, -6, -53, -151, -284, -301, 278, 2482, 7717, \ldots.$$  This sequence has a Hankel transform that begins
$$1, 0, -1, -1, 1, -1, -2, 1, 3, 5,\ldots,$$
where the sequence $$1, 1, -1, 1, 2, -1, -3, -5,\ldots$$
which is \seqnum{A006769} is the elliptic divisibility sequence \cite{Ward} associated with the elliptic curve
$$E: y^2-y=x^3-x.$$
This is a $(1,1)$ Somos $4$ sequence. It is the Hankel transform of the expansion of
$$f(x)=\frac{2x}{\sqrt{1-4x^3+4x^4}+2x-1}.$$
This expansion begins
$$1, 0, 1, -1, 1, -1, 0, 0, 0, -2, 4, -4, -1, 11, -16,\ldots.$$
\end{example}
We note finally that the generating function
$$g(x)=\frac{1}{1-(2-a)x-(1-a)x^2}c\left(\frac{-x^2((1-a)-(1-3a+a^2)x)}{(1-(2-a)x-(1-a)x^2)^2}\right)$$
can be put in the form of the continued fraction \cite{CFT, Wall}
$$
\cfrac{1}{1+(a-2)x+(a-1)x^2-
\cfrac{x^2(a-1+(1-3a+a^2)x)}{1+(a-2)x+(a-1)x^2-\cdots}}.$$
\section{Conclusions}
In this note, we have exhibited a three parameter family of involutory generating functions defined by the $B$-sequence with generating function
 $$B=\frac{a-cx}{1+bx}.$$ A special feature of this family is that, via Hankel transforms, it is closely linked to Somos $4$ sequences. In turn, these Somos sequences are linked to elliptic curves. We have shown that it is possible in certain circumstances to start with an elliptic curve, and by a sequence of transformations, arrive at an involutory power series. In particular, we have shown that the one-parameter family of elliptic curves $E: y^2- a xy-y=x^3-x$ gives rise to a corresponding family of Bell pseudo-involutions in the Riordan group.
\section{Acknowledgement}
Many of the techniques used in this paper are based on investigations into elliptic curves and the fascinating Somos sequences, themselves originating in the elliptic divisibility sequences \cite{Ward}, and further elaborated by Michael Somos, whose creative mathematics and many relevant contributions to the Online Encyclopedia of Integer Sequences \cite{SL1, SL2} have been inspirational.

\bigskip
\hrule
\bigskip
\noindent 2010 {\it Mathematics Subject Classification}: Primary
11B83; Secondary 11C20, 11B37, 14H52, 15B05, 15B36.
\noindent \emph{Keywords:} Riordan array, Riordan pseudo-involution, B-sequence, A-sequence, elliptic curve, Somos sequence, recurrence,  Hankel transform.

\bigskip
\hrule
\bigskip
\noindent (Concerned with sequences
\seqnum{A000108},
\seqnum{A000245},
\seqnum{A004148},
\seqnum{A007477},
\seqnum{A006196},
\seqnum{A006769},
\seqnum{A023431},
\seqnum{A025227},
\seqnum{A025243},
\seqnum{A025250},
\seqnum{A025258},
\seqnum{A025273},
\seqnum{A050512},
\seqnum{A060693},
\seqnum{A068875},
\seqnum{A086246},
\seqnum{A089796},
\seqnum{A090181},
\seqnum{A091561},
\seqnum{A091565},
\seqnum{A105633},
\seqnum{A130749},
\seqnum{A152225},
\seqnum{A178075},
\seqnum{A178075},
\seqnum{A178622},
\seqnum{A178622},
\seqnum{A178627},
\seqnum{A187256}, and
\seqnum{A217333}.)


\begin{thebibliography}{99}

\bibitem{Book} P. Barry, \emph{Riordan Arrays: A Primer}, Logic Press, 2017.

\bibitem{Gen} P. Barry, Generalized Catalan numbers, Hankel transforms and Somos-$4$ sequences,
\emph{J. Integer Seq.}, \textbf{13} (2010), \href{https://cs.uwaterloo.ca/journals/JIS/VOL13/Barry1/barry95r.pdf} {Article 10.7.2}.

\bibitem{CFT} P. Barry, Continued fractions and
transformations of integer sequences, \emph{J. Integer Seq.}, \textbf{12} (2009),
\href{https://cs.uwaterloo.ca/journals/JIS/VOL12/Barry3/barry93.pdf} {Article 09.7.6}.

\bibitem{PSI} N. T. Cameron and A. Nkwanta, On some (pseudo) involutions in the Riordan group, \emph{J. Integer Seq.}, \textbf{8} (2005),
\href{https://cs.uwaterloo.ca/journals/JIS/VOL8/Cameron/cameron46.pdf} {Article 05.3.7}.

\bibitem{Conj} X-K. Chang and X-B. Hu, A conjecture based on Somos-$4$ sequence and its extension,
\emph{Linear Algebra Appl.}, \textbf{436} (2012), 4285--4295.

\bibitem{B_Seq} G-S. Cheon, S-T. Jin, H. Kim, and L. W. Shapiro, Riordan group involutions and the $\Delta$-sequence, \emph{Disc. Applied Math.}, \text{157} (2009), 1696--1701.

\bibitem{Inv} G-S. Cheon and H. Kim, Simple proofs of open problems about the structure of involutions in the Riordan group, \emph{Linear Algebra Appl.}, \textbf{428} (2008), 930--940.

\bibitem{Cohen} M. M. Cohen, Elements of Finite Order in the Riordan Group, arXiv:1806.06432v1 [math.CO].

\bibitem{Luzon} A. Luz\'on, M. A. Mor\'on, and L. F. Prieto-Martinez, The group generated by Riordan involutions, arXiv: 1803.06872v1 [math.GR].

\bibitem{Candice} C. A. Marshall, Construction of Pseudoinvolutions in the Riordan group, Morgan State University Dissertation, 2017.

\bibitem{Cons} D. Phulara and L. Shapiro, Constructing Pseudo-involutions in the Riordan group, \emph{J. Integer Seq.}, \textbf{20} (2017), \href{https://cs.uwaterloo.ca/journals/JIS/VOL20/Phulara/phulara3.pdf} {Article 17.4.7}.

\bibitem{Survey} L. Shapiro, A survey of the Riordan group, available electronically at
\href{http://www.combinatorics.cn/activities/Riordan\%20Group.pdf} {Center for Combinatorics}, Nankai University, 2018.

\bibitem{SGWW} L. W. Shapiro, S. Getu, W.-J. Woan, and L. C.
    Woodson,
The Riordan group, \emph{Discr. Appl. Math.} \textbf{34}
(1991), 229--239.

\bibitem{SL1} N. J. A.~Sloane, \emph{The
On-Line Encyclopedia of Integer Sequences}. Published electronically
at \href{http://oeis.org}{http://oeis.org}, 2018.

\bibitem{SL2} N. J. A.~Sloane, The On-Line Encyclopedia of Integer
Sequences, \emph{Notices Amer. Math. Soc.}, \textbf{50} (2003),  912--915.

\bibitem{Wall} H.~S. Wall, \emph{Analytic Theory of
    Continued Fractions}, AMS Chelsea Publishing, 2001.
    
\bibitem{Ward} M. Ward, Memoir on elliptic divisibility sequences, \emph{Amer. J. Math.}, \textbf{70} (1948), 31--74.

\end{thebibliography}
\end{document}